\documentclass[11pt]{article}
\usepackage[latin9]{inputenc}
\usepackage{amsmath}
\usepackage{amssymb}
\usepackage{esint}

\makeatletter

\usepackage{cite}
\usepackage{amsthm}
\usepackage{amsmath,amssymb,amsfonts}
\usepackage{textcomp}
\usepackage{blindtext, graphicx}
\usepackage{balance}
\usepackage{graphicx,color}
\usepackage{epsfig}
\usepackage{psfrag}
\usepackage{booktabs} 
\usepackage{enumitem}
\usepackage{algpseudocode,algorithm,algorithmicx}
\usepackage{graphicx}
\usepackage{caption}
\usepackage{subcaption}
\usepackage{optidef}
\usepackage{mdframed}
\usepackage{tikz} 

\newcommand{\R}{\mathbb{R}}

\newcommand{\LG}{\mathbb{LG}}

\newcommand{\N}{\mathbb{N}}

\newcommand{\SO}{\mathbb{SO}}
\newcommand{\SE}{\mathbb{SE}}
\newcommand{\GL}{\mathbb{GL}}
\newcommand{\SL}{\mathbb{SL}}

\newcommand{\cA}{\mathcal{A}}
\newcommand{\cB}{\mathcal{B}}
\newcommand{\cBs}{\mathcal{BS}}

\newcommand{\cH}{\mathcal{H}}
\newcommand{\cG}{\mathcal{G}}
\newcommand{\cLG}{\mathcal{LG}}
\newcommand{\cT}{\mathcal{T}}

\newcommand{\cD}{\mathcal{D}}
\newcommand{\cE}{\mathcal{E}}
\newcommand{\cK}{\mathcal{K}}

\newcommand{\cV}{\mathcal{V}}

\newcommand{\cGL}{\mathcal{GL}}
\newcommand{\cSL}{\mathcal{SL}}
\newcommand{\cSO}{\mathcal{SO}}
\newcommand{\cSE}{\mathcal{SE}}

\def\v{\mathbf{v}}

\def\x{\mathbf{x}}

\def\ze{\mathbf{0}}

\DeclareMathOperator*{\Span}{span}

\newtheorem{theorem}{Theorem}[section]
\newtheorem{proposition}{Proposition}[section]

\newtheorem{definition}{Definition}[section]

\newtheorem{lemma}{Lemma}[section]
\newtheorem{example}{Example}[section]
\newtheorem{remark}{Remark}[section]
\newtheorem{problem}{Problem}

\usepackage{anysize}
\marginsize{2.5cm}{2cm}{1.5cm}{1.5cm}

\makeatother

\begin{document}
  \begin{center} \section*{Structural Controllability of Bilinear Systems on $\SE(n)$}{A. Sanand Amita Dilip \footnote{Department of Electrical Engineering, IIT Kharagpur, email: sanand@ee.iitkgp.ac.in}, Chirayu D. Athalye\footnote{Department of Electrical \& Electronics Engineering, BITS Pilani, K.K. Birla Goa Campus, email: chirayua@goa.bits-pilani.ac.in}}\\
 
 \end{center}
  
\begin{abstract}                          
Structural controllability challenges arise from imprecise system modeling and system interconnections in large scale systems. 
In this paper, we study structural control of bilinear systems on the special Euclidean group.
We employ graph theoretic methods to analyze the structural controllability problem for driftless bilinear systems and structural accessibility for bilinear systems with drift.
This facilitates the identification of a sparsest pattern necessary for achieving structural controllability and discerning redundant connections.
To obtain a graph theoretic characterization of structural controllability and accessibility on the special Euclidean group, we introduce a novel idea of solid and broken edges on graphs; subsequently, we use the notion of transitive closure of graphs.
 \end{abstract}
  
\section{Introduction}
\label{sec:intro}

Over the past decade and a half, driven by the advancement of modern computational tools and communication technologies, significant strides have been made in research on networked control systems.
When modeling such large scale systems, the exact numerical entries of the system models are often unknown.
For example, in networked systems, the precise weights of links influencing the strength of interactions are typically not known.
The consequence of even a minor numerical perturbation in a model could be the potential loss of some system properties.
To address such challenges, it is important to analyze the structural properties \cite{Ramos} of the underlying system. 
Since networked control systems are becoming increasingly prevalent, examining their structural properties is critical in both design and analysis problems.
Structural controllability is one such property which is in fact generic, i.e., true for almost all numerical realizations \cite{Boothby1,Lin}.
While studying structural properties of systems, the entries in the respective models are classified into zero and nonzero entries.
The zero and nonzero patterns in the model might be due to the inherent structure of the underlying system or the interconnection of subsystems present in the original system, e.g., multiagent systems \cite{Liu}.
We note that the presence or absence of an interconnecting link is more important than its numerical value. 

Structural controllability for LTI systems was first introduced in \cite{Lin}, where a graph theoretic characterization was obtained.
Graph theoretic methods provide an elegant way of analyzing such structural properties where the construction of the graph depends solely on the zero/nonzero pattern in the system matrices.
We refer the reader to \cite{Ramos,Dion} and the references therein for a comprehensive exploration of structural controllability.
Controllability of nonlinear systems was first characterized in seminal works of \cite{JurdSuss1,JurdSuss}.
Some early works on the control of bilinear systems include \cite{Boothby1,Boothby,Brockett}; for more details, we refer the reader to \cite{Elliot}.

Apart from the commonly employed linear models \cite{Ramos,Dion}, bilinear systems represent an important class of models used to study large scale systems \cite{Tsope,Sarlette} with applications in diverse areas like biology \cite{Williamson}, quantum control \cite{Boussaid}, economics \cite{aoki} and so on. 
Bilinear systems on the special Euclidean group have appeared in synchronization and coordination problems \cite{Sarlette} as well as in the distributed control of rigid bodies \cite{Ibuki}.
Notably, graph theoretic methods have found applications in both structural and nonstructural properties of bilinear systems \cite{Tsope,cz,Wang,Zhang}.
In \cite{Tsope}, driftless bilinear systems on $\R^n$ of the following form are considered: 
\begin{equation}
\label{eq:bilinear-on-R^n}
\dot{\x}(t) = \bigg(\sum_{i=1}^{m}u_i(t)B_i\bigg)\x(t), 
\end{equation}
where $B_1,\ldots,B_m$ are $(n\times n)$ structured matrices belonging to the Lie algebra $\cGL(n)$, $\cSL(n)$ of Lie groups $\GL(n)$, $\SL(n)$, respectively.
The necessary and sufficient conditions for structural controllability and accessibility were obtained using graph theoretic methods.
The main advantage of the graph theoretic equivalent conditions for controllability is the reduced computational cost of checking the Lie algebraic rank conditions.

\subsection{Motivation and Contribution}
Structural controllability serves to elucidate the essential interconnections necessary for system controllability, distinguishing them from redundant connections in the overall structure, particularly in large scale systems.
Furthermore, in the design problems over $\SE(n)$ where the exact system parameters are unknown, the inherent generality of structural controllability becomes pivotal in achieving the desired objectives.
In this paper, building on the insights gained from \cite{Tsope}, we analyze the structural controllability problem for driftless bilinear systems on $\SE(n)$.
However, note that this extension of results in \cite{Tsope} is not straightforward due to some subtleties in the Lie group structure of $\SE(n)$.
We introduce some novel graph constructions (Definitions \ref{def:pattern-graph} and \ref{def:transitive-closure}) to obtain graph theoretic characterizations for structural controllability and accessibility on $\SE(n)$ (Theorems \ref{thm:s-controllability} and \ref{thm:s-accessibility}).
Subsequently, algorithms in graph theory \cite{West,klein} can be effectively applied to check the structural controllability.
This helps in identifying superfluous connections which do not impact controllability. 
As a result, we obtain a sparsest pattern subject to maintaining structural controllability.

\subsection{Organization}
This paper is organized as follows. In Section \ref{sec:prob-statement}, we cover preliminaries and discuss the problem statement.
Graph theoretic characterizations of structural controllability and accessibility for bilinear systems on $\SE(n)$ is given in Section \ref{sec:main-results}.
Subsequently, in Section \ref{sec:link-failures}, we use the graph theoretic characterization to obtain a sparsest pattern for structural controllability.
Finally, we conclude the paper in Section \ref{sec:conclusion}.
Some auxiliary results required to prove the main results are deferred to Appendix.

\subsection{Notation}
The set of positive integers is denoted by $\N$.
An abstract matrix Lie group is denoted by $\LG$ and the corresponding matrix Lie algebra by $\cLG$.
We use $\GL(n)$, $\SL(n)$, $\SO(n)$, $\SE(n)$ to denote the general linear, the special linear, the special orthogonal, and the special Euclidean groups; whereas $\cGL(n)$, $\cSL(n)$, $\cSO(n)$, $\cSE(n)$ are used to denote their respective Lie algebras. 
$\cG(\cV,\cE)$ denotes a graph with $\cV$ as the vertex set and $\cE$ as the edge set.
The complete graph on $n$ vertices is denoted by $\cK_n$. The space of structured matrices with nonzero pattern $\Lambda$ is denoted by $\cB_\Lambda$.
The entries of structured matrices taking arbitrary real values are denoted by $\ast$. 

\section{Problem Statement and Preliminaries}
\label{sec:prob-statement}

\subsection{Controllability and Accessibility of Bilinear Systems}
\label{subsec:background}
In this subsection, we succinctly develop a background needed to discuss the problem statement.
Consider a driftless bilinear control system of the form: 
\begin{equation}
\label{eq:driftless-bilinear}
\dot{X}(t) = \bigg(\sum_{i=1}^{m}u_i(t)B_i\bigg)X(t),\quad X(0)=X_0,
\end{equation}
where $X(\cdot) \in \LG$ and $B_1,\ldots,B_m \in \cLG$.
The Lie algebra $\cLG$ is the tangent space to a matrix Lie group $\LG$ at the identity element. More specifically, $\cLG$ is a vector space closed under the Lie bracket operation defined as 
\begin{equation*}
\big[A_1,A_2\big] := A_1A_2-A_2A_1 \qquad \forall A_1,A_2\in\cLG.
\end{equation*}
If $A_1,\ldots,A_k\in \cLG$, then the vector space constructed using the repeated Lie bracket operation on the set $\{A_1,\ldots,A_k\}$ is called the Lie subalgebra generated by $\{A_1,\ldots,A_k\}$.
We refer readers to \cite{Hall} for more details on Lie groups and Lie algebras. 

Next, we briefly cover the concept of controllability for \eqref{eq:driftless-bilinear}; 
see \cite{JurdSuss1,JurdSuss} for more details.
A point $Y\in \LG$ is attainable from $X(0)=X_0\in \LG$ at $t_1 \geq 0$ if there exists $u_1(\cdot),\ldots,u_m(\cdot)$ such that $X(t_1)=Y$.
The set of attainable points from $X_0$ is given by
\begin{equation}
\label{eq:attainable-set}
\cA(X_0) := \bigcup_{t \in [0,\infty)}\cA(X_0,t),
\end{equation}
where $\cA(X_0,t)$ is the set of attainable points from $X_0$ at time $t$.
The system \eqref{eq:driftless-bilinear} is said to be {\em controllable from $X_0$} if $\cA(X_0)=\LG$, and if it is controllable from every $X_0\in \LG$, then we say that \eqref{eq:driftless-bilinear} is controllable.
For the ease of reference later, we state below \cite[Th. 7.1]{JurdSuss} which gives an equivalent condition for controllability of a driftless bilinear system. 
This condition is known as the Lie Algebraic Rank Condition (LARC). 
\begin{theorem}[\!\!\cite{JurdSuss}]
\label{thm:LARC}
The driftless bilinear system given by \eqref{eq:driftless-bilinear} is controllable if and only if $\LG$ is connected and the Lie subalgebra generated by $\{B_1,\ldots,B_m\}$ is equal to $\cLG$. 
\end{theorem}

Now, consider a bilinear system with drift on $\LG$:
\begin{equation}
\label{eq:drift-bilinear}
\dot{X}(t) = B_0 X(t) + \bigg(\sum_{i=1}^{m}u_i(t)B_i\bigg)X(t), \quad X(0) = X_0,
\end{equation}
where $B_0,B_1,\ldots,B_m \in \cLG$.
Controllability and accessibility of bilinear systems with drift have been studied from graph theoretic viewpoint in \cite{Wang} on Lie groups $\GL(n),\SL(n),\SO(n)$.
We briefly cover the concept of accessibility \cite{JurdSuss,JurdSuss1} for bilinear systems with drift on an abstract Lie group. 
The system \eqref{eq:drift-bilinear} is said to be {\em accessible from $X_0$} if the set of attainable points $\cA(X_0)$ given by \eqref{eq:attainable-set} has a nonempty interior, and it is said to be accessible if $\cA(X_0)$ has a nonempty interior for all $X_0\in \LG$. The following well-known result characterizes accessibility for bilinear systems with drift.

\begin{proposition}[\!\!\cite{JurdSuss1,JurdSuss}]
\label{prop:accessibility}
The bilinear system with drift given by \eqref{eq:drift-bilinear} is accessible if and only if $\LG$ is connected and the Lie subalgebra generated by $\{B_0,B_1,\ldots,B_m\}$ is equal to $\cLG$.  
\end{proposition}

\subsection{Structural Controllability and Accessibility}
The set of structured matrices on $\cLG\subseteq \R^{n\times n}$ corresponding to $\Lambda \subseteq \{1,2,\ldots,n\}\times\{1,2,\ldots,n\}$ is given by
\begin{align*}
\label{eq:struct_matrices_Lie_alg}
\cB_\Lambda := \big\{ B\in \cLG \mid &\, B(i,j) =\, \ast \mbox{ if }(i,j)\in\Lambda \mbox{ and } \nonumber\\ 
&\, B(i,j)= 0 \mbox{ otherwise} \big\}.
\end{align*}
Notice that the entries indexed by $\Lambda$ can be arbitrary real numbers which comply with the properties of $\cLG$, while the remaining entries are set to zero.
The choice of $\Lambda$ determines the nonzero pattern of $\cB_\Lambda$.
The set $\Lambda$ is intricately linked to the interaction pattern among the state variables in complex systems. 
This allows for a flexible and dynamic exploration of the structural properties of the system based on the specified interaction pattern.
\begin{definition}[Structural Controllability]
\label{def:s-controllability}
The pattern $\cB_\Lambda \subseteq \cLG$ is said to be structurally controllable if there exists $m\in\N$ and $B_1,\ldots,B_m \in \cB_\Lambda$ such that \eqref{eq:driftless-bilinear} is controllable.
\end{definition}
In fact, if there exists $m\in\N$ which satisfies the condition in Definition \ref{def:s-controllability}, then $\cB_\Lambda$ is also said to be {\em $m$-structurally controllable}. 
Note that for $\cB_\Lambda$ to be $m$-structurally controllable, the LARC condition given by Theorem \ref{thm:LARC} must be satisfied for at least one numerical realization of $B_1,\ldots,B_m \in \cB_\Lambda$.
In \cite{Tsope}, graph theoretic necessary and sufficient conditions for the structural controllability of driftless bilinear systems \eqref{eq:bilinear-on-R^n} were obtained for $\cB_\Lambda\subseteq \cGL(n)$ and $\cB_\Lambda\subseteq \cSL(n)$.   
Furthermore, it was shown that the structural controllability for bilinear systems of the form \eqref{eq:bilinear-on-R^n} is a generic property \cite[Th. III.1]{Tsope}, which means it is true for almost all numerical realizations of $B_1,\ldots,B_m \in \cB_\Lambda$. 
On similar lines, it can be shown that the structural controllability of \eqref{eq:driftless-bilinear} is also generic.
\begin{definition}[Structural Accessibility]
\label{def:s-accessibility}
The pattern $\cB_\Lambda \subseteq \cLG$ is said to be structurally accessible if there exists $m \in \N$ and $B_0,B_1,\ldots,B_m\in \cB_\Lambda$ such that \eqref{eq:drift-bilinear} is accessible. 
\end{definition}
Similar to structural controllability, if there exists $m\in \N$ satisfying the condition in Definition \ref{def:s-accessibility}, then $\cB_\Lambda$ is called {\em $m$-structurally accessible}.

In this paper, we analyze the structural controllability and structural accessibility for bilinear systems \eqref{eq:driftless-bilinear} and \eqref{eq:drift-bilinear}, respectively, on $\SE(n)$. The special Euclidean group $\SE(n)$ consists of distance and orientation preserving transformations on $\R^n$, and it is a connected Lie group.
In particular, a special Euclidean transformation $T$ on $\R^n$ can be defined as $T(\x) := Q\x+\v$ where $Q \in \SO(n)$ and $\v\in\R^n$; in other words, it is a rotation followed by a translation.
This can be expressed as a linear transformation on $\R^{n+1}$ as follows:
\begin{equation*}
\begin{bmatrix} Q & \v \\ \ze^\top & 1 \end{bmatrix} \begin{bmatrix} \x \\ 1 \end{bmatrix} = \begin{bmatrix} Q\x+\v \\ 1 \end{bmatrix}.
\end{equation*}
This gives a representation of $\SE(n)$ as a matrix group in $\R^{(n+1)\times (n+1)}$.
Let $\Omega_{ij} \in \R^{n \times n}$ denote the skew-symmetric matrix with $1$ in the $(i,j)^{\rm th}$ entry, $-1$ in the $(j,i)^{\rm th}$ entry, and zero elsewhere.
Let
\begin{equation}
\label{eq:Omega-tilde}
\tilde{\Omega}_{ij} = \begin{bmatrix} \Omega_{ij} & \ze \\ \ze^\top & 0 \end{bmatrix} \in \R^{(n+1)\times(n+1)},
\end{equation}
and $E_{ij} \in \R^{(n+1)\times(n+1)}$ denote the matrix with $1$ in the $(i,j)^{\rm th}$ entry and zero elsewhere.
The Lie algebra $\cSE(n)$ of the Lie group $\SE(n)$ is a subspace of $\R^{(n+1)\times(n+1)}$ having dimension $n(n+1)/2$.
The standard basis of $\cSE(n)$ is given by 
\begin{equation}
\label{eq:std-basis_SE(n)} 
\cBs := \big\{\tilde{\Omega}_{ij} \mid 1 \le i<j \le n \big\} \cup \big\{E_{k(n+1)} \mid 1\le k\le n \big\}.
\end{equation}
Notice that for matrices in $\cSE(n)$, the $(n+1)^{\rm th}$ row is identically zero.

\subsection{Problem Formulation}
\label{subsec:problems}
Let $\cB_\Lambda$ be a set of structured matrices in $\cSE(n)$ with locations of nonzero entries given by $\Lambda \subseteq \{1,2,\ldots,n\}\times\{1,2,\ldots,n+1\}$. 
We consider the following structured bilinear systems:
\begin{equation}
\label{eq:driftless-bilinear-SE(n)}
\dot{X}(t) = \bigg(\sum_{i=1}^{m}u_i(t)B_i\bigg)X(t),\quad X(0)=X_0,
\end{equation} 
\begin{equation}
\label{eq:drift-bilinear-SE(n)}
\dot{X}(t) = B_0 X(t) + \bigg(\sum_{i=1}^{m}u_i(t)B_i\bigg)X(t), \quad X(0) = X_0,
\end{equation}
where $X(\cdot) \in \SE(n)$ and $B_1,\ldots,B_m \in \cB_\Lambda \subseteq \cSE(n)$. Note that for a given $\Lambda$, we have $\cB_\Lambda = \Span(\cBs_{\Lambda})$, where
\begin{align}
\cBs_{\Lambda} := \big\{\tilde{\Omega}_{ij} \mid 1 \le i<j \le n \mbox{ and } (i,j) \in \Lambda &\big\} \;\cup \nonumber \\
\big\{E_{k(n+1)} \mid (k,n+1)\in \Lambda &\big\} \label{eq:std-basis_B-lambda} 
\end{align}
is the standard basis for $\cB_\Lambda$.

In this paper, we address the following two problems.
\begin{problem}
\label{prob-1}
Find graph theoretic equivalent conditions for the structural controllability of \eqref{eq:driftless-bilinear-SE(n)} and structural accessibility of \eqref{eq:drift-bilinear-SE(n)}.
\end{problem}
\begin{problem}
\label{prob-2}
Minimize $|\Lambda|$ subjected to the structural controllability of \eqref{eq:driftless-bilinear-SE(n)}.
\end{problem}
The equivalent conditions obtained for Problem \ref{prob-1} allow us to identify the essential interconnections for structural controllability and accessibility. 
Subsequently, these conditions are leveraged to solve Problem \ref{prob-2} which deals with the situation where there is a uniform cost associated with each nonzero entry in the pattern $\cB_\Lambda$. Each of these nonzero entries corresponds to an interconnection between the state variables. 
Since the cost associated with each nonzero entry is the same, the total cost is indeed proportional to $|\Lambda|$ which measures the sparsity of the pattern $\cB_\Lambda$.
Thus, Problem \ref{prob-2} is naturally applicable in the design of large-scale systems, where the aim is to achieve structural controllability with minimum interconnections. 
Notice that using $\cB_\Lambda \subseteq \cSE(n)$, the corresponding uniform cost matrix $C \in \R^{(n+1) \times (n+1)}$ can be defined as follows:
\begin{equation*}
\label{eq:uniform-cost-matrix}
C(i,j) := 
\begin{cases}
1, & \text{ if }  i\neq j \mbox{ and } 1\le i\le n,  \\
0, & \text{ otherwise.}
\end{cases}
\end{equation*}
Recall that for matrices in $\cSE(n)$, the $(n+1)^{\rm th}$ row is identically zero.
We also discuss the above problem under non-uniform costs.

\subsection{Preliminaries}
\label{subsec:preliminaries}
We now introduce some tools starting with basic graph theory which is used in Section \ref{sec:main-results} to obtain equivalent conditions for structural controllability.
An undirected graph $\cG(\cV,\cE)$ consists of a vertex set $\cV$ and an edge set $\cE$.
An edge connecting vertices $i,j\in \cV$ is denoted as $(i,j)\in \cE$. 
The degree of a vertex is the number of edges incident on that vertex. 
A simple graph is a graph without loops and multiple edges.
A graph is said to be complete if every pair of distinct vertices is connected by a unique edge.
In a directed graph (digraph), the order of vertices in a pair $(i,j)$ determines the direction of an edge.
A path between two vertices $i_1,i_k$ is the sequence of vertices $\{i_1,i_2,\ldots,i_k\}$ such that $(i_j,i_{j+1})\in \cE$ for $j=1,\ldots,k-1$.
An undirected graph is said to be connected if there exists a path between every pair of vertices, otherwise it is said to be disconnected.
Next, we define a graph structure which is used in Section \ref{sec:main-results} for obtaining graph theoretic equivalent conditions for the structural controllability of \eqref{eq:driftless-bilinear-SE(n)}. 
\begin{definition}
\label{def:pattern-graph}
For $\cB_\Lambda \subseteq \cSE(n)$, we associate the undirected simple graph $\cG\big(\cB_\Lambda\big)$ with the vertex set $\cV = \{1,\ldots,n+1\}$ and the edge set $\cE = \cE_s \cup \cE_b$, where
\begin{align*}
\cE_s &:= \big\{(i,j) \mid  \tilde{\Omega}_{ij} \in \cBs_\Lambda \big\}, \\
\cE_b &:= \big\{(k,n+1) \mid E_{k(n+1)} \in \cBs_\Lambda \big\},
\end{align*}
and $\cBs_\Lambda$ given by \eqref{eq:std-basis_B-lambda} is the standard basis of $\cB_\Lambda$.
The edges in $\cE_s$ are denoted by solid lines, and $\cE_s$ is called the set of solid edges. 
Whereas the edges in $\cE_b$ are denoted by broken lines, and $\cE_b$ is called the set of broken edges.
\end{definition}
Inspired from \cite{Tsope}, we define below the transitive closure of undirected simple graphs considered in Definition \ref{def:pattern-graph}.
\begin{definition}
\label{def:transitive-closure}
Let $\cG(\cV,\cE)$ be an undirected simple graph corresponding to $\cB_\Lambda \subseteq \cSE(n)$.
The $l^{\rm th}$ transitive closure of $\cG$, denoted by $\cG^{(l)}$, is defined as follows.
Let $\cG^{(0)} = \cG$ with $\cE^{(0)} = \cE_s \cup \cE_b$.
\begin{algorithmic}
\For{$1\le l \le n$}
\If{$(i,j),(j,k)\in \cE^{(l-1)}_s$}  $(i,k)\in \cE^{(l)}_s$.
\ElsIf{$(i,j),(j,k)\in \cE^{(l-1)}$, where one edge belongs to $\cE^{(l-1)}_s$ and the other belongs to $\cE^{(l-1)}_b$} $(i,k)\in \cE^{(l)}_b$.
\ElsIf{$(i,j),(j,k)\in \cE^{(l-1)}_b$} $(i,k)\notin \cE^{(l)}$.
\EndIf
\EndFor
\end{algorithmic}
The $n^{\rm th}$ transitive closure of $\cG(\cV,\cE)$ is the transitive closure. 
\end{definition}

We mention some standard Lie bracket identities satisfied by elements of $\cBs$ given by \eqref{eq:std-basis_SE(n)}
\begin{subequations}
\label{eq:ELA-LB}
\begin{align}
& \big[\tilde{\Omega}_{ij},\tilde{\Omega}_{kl} \big] = \delta_{jk}\tilde{\Omega}_{il}+\delta_{il}\tilde{\Omega}_{jk}+\delta_{jl}\tilde{\Omega}_{ki}+\delta_{ik}\tilde{\Omega}_{lj}, \label{eq:ELA-LB-1} \\
& \big[E_{i(n+1)},E_{j(n+1)} \big] = 0, \label{eq:ELA-LB-2} \\
& \big[\tilde{\Omega}_{ij},E_{k(n+1)} \big] = \delta_{jk} E_{i(n+1)} - \delta_{ik} E_{j(n+1)} \label{eq:ELA-LB-3}
\end{align}
\end{subequations}
where 
\begin{equation*}
\label{eq:def_K-delta}
\delta_{ij} := \begin{cases}
1, & \mbox{ if }  i=j, \\
0, & \mbox{ otherwise,} 
\end{cases}
\end{equation*}
is the Kronecker delta function.
In particular, \eqref{eq:ELA-LB-1} follows trivially from \cite[Lem. 2.3]{cz}; whereas \eqref{eq:ELA-LB-2} and \eqref{eq:ELA-LB-3} follow from construction. The above identities and the concept of derived distributions \cite[Sec. III]{Tsope}, which is explained next, are used in the next section to obtain graph theoretic equivalent conditions for structural controllability and accessibility of bilinear systems on $\SE(n)$.
Given a subspace $\cD^{(0)}\subseteq \cSE(n)$, the $i^{\rm th}$ derived distribution of $\cD^{(0)}$ is given by
\begin{equation}
\label{eq:derived-distribution}
\cD^{(i)}:=\cD^{(i-1)} \oplus \Span \!\big\{[A_1,A_2] \mid A_1,A_2\in \cD^{(i-1)} \big\}.
\end{equation}
The Lie subalgebra generated by $\cD^{(0)}$, using derived distributions, is denoted by $\overline{\cD}$.

Suppose $\cD^{(0)} =\cB_\Lambda \subseteq \cSE(n)$ is a {\em canonical} subspace\footnote{A subspace of $\cSE(n)$ formed by the span of a subset of the standard basis $\cBs$.} of $\cSE(n)$. 
Then for an $i^{\rm th}$ derived distribution $\cD^{(i)}$, we can associate the pattern $\cB_{\Lambda^{(i)}}$ of structured matrices; this follows from \eqref{eq:ELA-LB} and \eqref{eq:derived-distribution}.
Thus, for the subspace $\cD^{(i)}$, we can associate an undirected graph $\cG\big(\cD^{(i)}\big)$ using Definition \ref{def:pattern-graph}.  
Conversely, for an $i^{\rm th}$ transitive closure $\cG^{(i)}$ of $\cG\big(\cB_\Lambda\big)$, we can associate a subspace $\cB\big(\cG^{(i)}\big) \subseteq \R^{(n+1)\times (n+1)}$ of structured matrices as follows: 
\begin{equation}
\label{eq:subspace-B}
\cB\big(\cG^{(i)}\big)(j,k) = 
\begin{cases}
\ast, & \text{if } (j,k) \in \cG^{(i)} \mbox{ and }  j<k, \\
0, & \text{if } (j,k) \notin \cG^{(i)} \mbox{ and } j<k, \\
0, & \text{if } j=k \mbox{ or } j=n+1, 
\end{cases}
\end{equation}
together with $\cB\big(\cG^{(i)}\big)(k,j)=-\cB\big(\cG^{(i)}\big)(j,k)$ for $1\le j<k\le n$.

\section{Main Results}
\label{sec:main-results}

The following theorem gives graph theoretic equivalent conditions for the structural controllability of \eqref{eq:driftless-bilinear-SE(n)}. Subsequently, it is extended to the structural accessibility (Theorem \ref{thm:s-accessibility}) for systems with drift on $\SE(n)$.
\begin{theorem}
\label{thm:s-controllability}
Consider \eqref{eq:driftless-bilinear-SE(n)} with $\cB_\Lambda \subseteq \cSE(n)$; let $\cG\big(\cB_\Lambda\big)$ be its associated graph. 
The following are equivalent:
\begin{enumerate}[label=(\roman*)]
\item The pattern $\cB_\Lambda$ is structurally controllable.
\item The transitive closure of $\cG\big(\cB_\Lambda\big)$ is the complete graph on $(n+1)$ vertices.
\item The graph $\cG\big(\cB_\Lambda\big)$ is connected on $\cV =\{1,\ldots,n+1\}$, and $\cG\big(\cB_\Lambda(1:n,1:n)\big)$ is also connected on vertices $\{1,\ldots,n\}$.
\end{enumerate}
\end{theorem}
\begin{proof}
 Since $\SE(n)$ is connected, we can conclude using Theorem \ref{thm:LARC} and Definition \ref{def:s-controllability} that $\cB_\Lambda$ is structurally controllable if and only if the Lie subalgebra generated by the subspace $\cB_\Lambda$ is equal to $\cSE(n)$.

$(i)\Rightarrow(ii)$: Since $\cB_\Lambda$ is structurally controllable, the Lie subalgebra generated by the subspace $\cB_\Lambda$ is $\cSE(n)$. 
By Lemma \ref{lem:core-lemma}, with $\cD^{(0)} =\cB_\Lambda$ and $\cG^{(0)} =\cG\big(\cB_\Lambda\big)$, we have $\overline{\cD}=\cD^{(n)}=\cSE(n)$ and $\cG\big(\cD^{(n)}\big) = \cG^{(n)}$. 
As a result, $\cG^{(n)}$ is the complete graph; recall from Definition \ref{def:transitive-closure} that $\cG^{(n)}$ is the transitive closure of $\cG\big(\cB_\Lambda\big)$.

$(ii)\Rightarrow(i)$: By Lemma \ref{lem:core-lemma}, with $\cD^{(0)} =\cB_\Lambda$ and $\cG^{(0)} =\cG\big(\cB_\Lambda\big)$, we have ${\cB}\big(\cG^{(n)}\big) = \cD^{(n)}$. Consequently, as $\cG^{(n)}$ is the complete graph on $(n+1)$ vertices, it follows from the construction of $\cB$ (see \eqref{eq:subspace-B}) that $\cD^{(n)}={\cB}\big(\cG^{(n)}\big)=\cSE(n)$. 
Since the the Lie subalgebra generated by $\cB_\Lambda$ is $\cSE(n)$, the pattern $\cB_\Lambda$ is structurally controllable. 

$(ii)\Leftrightarrow(iii)$: Follows from Definitions \ref{def:pattern-graph} and \ref{def:transitive-closure}. 

\end{proof}

The notion of solid and broken edges in Definitions \ref{def:pattern-graph} and \ref{def:transitive-closure} 
is primarily used for obtaining the equivalent conditions in Theorem \ref{thm:s-controllability}.
Therefore, the above result is not a consequence of \cite[Th. III.9]{Tsope}; this is further clarified in Remark \ref{rmk:auxiliary-results}.
It follows from Theorem \ref{thm:s-controllability} that to check the structural controllability of $\cB_\Lambda$, it is enough to check the connectedness of the associated graphs $\cG\big(\cB_\Lambda\big)$ and $\cG\big(\cB_\Lambda(1:n,1:n)\big)$. 
Moreover, note that the graph $\cG\big(\cB_\Lambda(1:n,1:n)\big)$ consist of only solid edges, and we need it to be connected for structural controllability.
However, as far as leveraging graph algorithms from the literature \cite{West} for checking the condition (iii) in Theorem \ref{thm:s-controllability} is concerned, it does not matter whether an edge is solid or broken. 
Checking the connectedness of an undirected graph can be done using depth first search or breadth first search algorithms \cite{klein} which require $O(|\cV|+|\cE|)$ time where $\cV,\cE$ are vertices and edges of $\cG\big(\cB_\Lambda\big)$.

\begin{example}
\label{ex:s-controllable system}
Consider the bilinear system \eqref{eq:driftless-bilinear-SE(n)} on $\SE(3)$ with the nonzero pattern $\Lambda =\{(1,2),(2,3),\\(1,4)\}$ which determines $\cB_\Lambda$.
By Definition \ref{def:transitive-closure}, the $l^{\rm th}$ transitive closure of $\cG\big(\cB_\Lambda\big)$ for $l=0,1,2$ are as shown in Figure \ref{fig:transitive-closure}.
\begin{figure}[ht]
\centering
\begin{subfigure}[b]{0.2\textwidth}
\centering
\begin{tikzpicture}[node distance={22mm}, thin, main/.style = {draw, circle}] 
\node[main,fill={rgb:red,0;green,100;blue,100},scale=0.60] (1) {$1$}; 
\node[main,fill={rgb:red,0;green,100;blue,100},scale=0.60] (4) [above right of=1] {$4$}; 
\node[main,fill={rgb:red,0;green,100;blue,100},scale=0.60] (2) [below right of =4] {$2$}; 
\node[main,fill={rgb:red,0;green,100;blue,100},scale=0.60] (3) [above of=4] {$3$}; 
\draw[line width=1.6pt] (1) -- (2); 
\draw[line width=1.6pt] (2) -- (3); 
\draw[line width=1.6pt] (1) -- (4)[dashed]; 
\end{tikzpicture} 
\caption{$\cG^{(0)}=\cG\big(\cB_\Lambda\big)$}
\end{subfigure} \hfill
\begin{subfigure}[b]{0.2\textwidth}
\centering
\begin{tikzpicture}[node distance={22mm}, thin, main/.style = {draw, circle}] 
\node[main,fill={rgb:red,0;green,100;blue,100},scale=0.60] (1) {$1$}; 
\node[main,fill={rgb:red,0;green,100;blue,100},scale=0.60] (4) [above right of=1] {$4$}; 
\node[main,fill={rgb:red,0;green,100;blue,100},scale=0.60] (2) [below right of =4] {$2$}; 
\node[main,fill={rgb:red,0;green,100;blue,100},scale=0.60] (3) [above of=4] {$3$}; 
\draw[line width=1.6pt] (1) -- (2); 
\draw[line width=1.6pt] (2) -- (3); 
\draw[line width=1.6pt] (1) -- (3); 
\draw[line width=1.6pt] (1) -- (4)[dashed]; 
\draw[line width=1.6pt] (2) -- (4)[dashed]; 
\end{tikzpicture} 
\caption{$\cG^{(1)}$}
\end{subfigure} \hfill
\begin{subfigure}[b]{0.2\textwidth}
\centering
\begin{tikzpicture}[node distance={22mm}, thin, main/.style = {draw, circle}] 
\node[main,fill={rgb:red,0;green,100;blue,100},scale=0.60] (1) {$1$}; 
\node[main,fill={rgb:red,0;green,100;blue,100},scale=0.60] (4) [above right of=1] {$4$}; 
\node[main,fill={rgb:red,0;green,100;blue,100},scale=0.60] (2) [below right of =4] {$2$}; 
\node[main,fill={rgb:red,0;green,100;blue,100},scale=0.60] (3) [above of=4] {$3$}; 
\draw[line width=1.6pt] (1) -- (2); 
\draw[line width=1.6pt] (2) -- (3); 
\draw[line width=1.6pt] (1) -- (3); 
\draw[line width=1.6pt] (1) -- (4)[dashed]; 
\draw[line width=1.6pt] (2) -- (4)[dashed]; 
\draw[line width=1.6pt] (3) -- (4)[dashed]; 
\end{tikzpicture}
\caption{$\cG^{(2)}$}
\end{subfigure}
\caption{$l^{\rm th}$ transitive closure of $\cG\big(\cB_\Lambda\big)$ in Example \ref{ex:s-controllable system}.}
\label{fig:transitive-closure}
\end{figure}
Since the transitive closure $\cG^{(2)}$ is the complete graph, it follows from Theorem \ref{thm:s-controllability} that $\cB_\Lambda$ is structurally controllable.
Also, notice the following:
\begin{enumerate}[label=(\roman*)]
\item Both $\cG\big(\cB_\Lambda\big)$ and $\cG\big(\cB_\Lambda(1:n,1:n)\big)$ are connected.
\item The Lie subalgebra formed by $B_1=\tilde{\Omega}_{12}$, $B_2=\tilde{\Omega}_{23}$ and $B_3=E_{14}$ is equal to $\cSE(3)$.
\end{enumerate}
\end{example}

\begin{example}
\label{ex:s-uncontrollable system}
Consider the bilinear system \eqref{eq:driftless-bilinear-SE(n)} on $\SE(3)$ with the nonzero pattern $\Lambda=\{(1,4),(3,4),\\(1,2)\}$.
\begin{figure}[ht]
\centering
\begin{tikzpicture}[node distance={22mm}, thin, main/.style = {draw, circle}] 
\node[main,fill={rgb:red,0;green,100;blue,100},scale=0.60] (1) {$1$}; 
\node[main,fill={rgb:red,0;green,100;blue,100},scale=0.60] (4) [above right of=1] {$4$}; 
\node[main,fill={rgb:red,0;green,100;blue,100},scale=0.60] (2) [below right of =4] {$2$}; 
\node[main,fill={rgb:red,0;green,100;blue,100},scale=0.60] (3) [above of=4] {$3$}; 
\draw[line width=1.6pt] (1) -- (2); 
\draw[line width=1.6pt] (3) -- (4)[dashed]; 
\draw[line width=1.6pt] (1) -- (4)[dashed]; 
\end{tikzpicture} 
\caption{$\cG\big(\cB_\Lambda\big)$ for the system in Example \ref{ex:s-uncontrollable system}.}
\label{fig:s-uncontrollable}
\end{figure}
Notice from Figure \ref{fig:s-uncontrollable} that $\cG\big(\cB_\Lambda\big)$ is connected, but $\cG\big(\cB_\Lambda(1:n,1:n)\big)$ is not connected. 
Therefore, by Theorem \ref{thm:s-controllability}, $\cB_\Lambda$ is structurally uncontrollable. 
\end{example}


In \cite{Tsope}, the problem of finding the minimum number of inputs for a structurally 
controllable pattern $\cB_{\Lambda}$ is considered; this is known as the minimum controllability problem.
A probability one algorithm is given in \cite{Tsope} to check if the system is controllable with $k$ inputs, $k\ge 2$ ($k$-controllability).
The same algorithm \cite[Algorithm 1]{Tsope} can be used to solve the minimum controllability problem for a structurally controllable pattern over $\SE(n)$.

Besides, structural controllability with multiple patterns is also studied in \cite{Tsope}, and results analogous to single pattern are obtained.
In this paper, we have studied only single pattern structured systems. However, using the concept of solid and broken edges, results in \cite{Tsope} can be extended for structural controllability on $\SE(n)$ for multiple patterns.

Recall from the proof of Theorem \ref{thm:s-controllability} that the Lie algebra generated by $\{B_0,B_1,\ldots,B_m\}$ is equal to $\cSE(n)$ if and only if the transitive closure of $\cG\big(\cB_{\Lambda}\big)$ is the complete graph.
Now, by Proposition \ref{prop:accessibility}, we get the following result which provides a graph based equivalent condition for structural accessibility.

\begin{theorem}
\label{thm:s-accessibility}
Consider \eqref{eq:drift-bilinear-SE(n)} with $B_0,\ldots,B_m\in\cB_{\Lambda}\subseteq \cSE(n)$; let $\cG\big(\cB_{\Lambda}\big)$ be its associated graph. 
Then $\cB_\Lambda$ is structurally accessible if and only if the transitive closure of $\cG\big(\cB_\Lambda\big)$ is the complete graph.
\end{theorem}

\section{Sparsity and Structural Controllability}
\label{sec:link-failures}
In this section, we discuss Problem \ref{prob-2}. 
We use the condition (iii) of Theorem \ref{thm:s-controllability} as an equivalent condition for structural controllability. 
The following theorem addresses Problem \ref{prob-2}.
\begin{theorem}
\label{thm:optimal s-controllability}
Let $\cT_{n+1}$ be the set of spanning trees on $(n+1)$ vertices such that if $\cG_{T}\in \cT_{n+1}$, then $\cG_{T}$ restricted to vertices $\{1,2,\ldots,n\}$ is a spanning tree.
Then, there is a one to one correspondence between the elements of $\cT_{n+1}$ and solutions of Problem \ref{prob-2}.
\end{theorem}
\begin{proof}
Let $\cG_{T}\in \cT_{n+1}$; there is a unique pattern $\hat{\Lambda}$ given by nonzero entries of $\cB\big(\cG_{T}\big)$. It follows from the definition of $\cT_{n+1}$ and Theorem \ref{thm:s-controllability} that $\hat{\Lambda}$ is a solution of Problem \ref{prob-2}.
 
Conversely, suppose $\hat{\Lambda}$ is a solution of Problem \ref{prob-2}, and
let ${\cG}\big(\cB_{\hat{\Lambda}}\big)$ be the associated graph. We show by contradiction that ${\cG}\big(\cB_{\hat{\Lambda}}\big)\in \cT_{n+1}$.
If ${\cG}\big(\cB_{\hat{\Lambda}}\big)$ or ${\cG}\big(\cB_{\hat{\Lambda}}(1:n,1:n)\big)$ strictly contains a spanning tree, we can drop some edges such that both ${\cG}\big(\cB_{\hat{\Lambda}}\big)$ and ${\cG}\big(\cB_{\hat{\Lambda}}(1:n,1:n)\big)$ contain a spanning tree. This contradicts that $\hat{\Lambda}$ is a sparsest pattern for which \eqref{eq:driftless-bilinear-SE(n)} is structurally controllable.
\end{proof}

Next, we consider the case of non-uniform cost associated with each nonzero entry in $\cB_\Lambda$.
Using the structure of $\cSE(n)$, we assign a non-uniform cost matrix $C\in \R^{(n+1)\times(n+1)}$ as follows:
\begin{equation}
\label{eq:cost_mat}
C(i,j) = 
\begin{cases}
c_{ij}>0, & \text{if } i<j, \\
0, & \text{if } i=j \mbox{ or } i=n+1, 
\end{cases}
\end{equation}
together with $C(j,i)=C(i,j)$ for $1\le i <j\le n$. Since the entries $\cB_\Lambda(i,j)$ and $\cB_\Lambda(j,i)$ are dependent for $1\le i,j\le n$, we have assigned the same cost for each such pair.
Once the cost matrix is defined, we can form a weighted undirected graph $\cG\big(\cB_\Lambda\big)$.
 
We consider the following problem: 
\begin{mini}
{\Lambda}{\sum_{(i,j)\in \Lambda} C(i,j)}{\label{eq:weighted_cost}}{}
\addConstraint{\mbox{ structural controllability of \eqref{eq:driftless-bilinear-SE(n)}.}}{}{\qquad}
\end{mini}
From the cost matrix $C$ in \eqref{eq:cost_mat}, we can form a weighted complete graph $\cK_{n+1}$ on $(n+1)$ vertices with the edge weights $w(i,j)= C(i,j)+C(j,i)$.
It follows from Theorem \ref{thm:optimal s-controllability} that the search space for \eqref{eq:weighted_cost} can be restricted to $\cT_{n+1}$.
We give a simple greedy algorithm (Algorithm 1) to solve \eqref{eq:weighted_cost}.
\begin{algorithm}
\caption{An algorithm to solve \eqref{eq:weighted_cost}}
\begin{algorithmic}[1]\label{algo:minspan}
\State $\cH_T \gets$ a minimum spanning tree on $\cK_{n}$ (the complete graph on the vertex set $\{1,\ldots,n\}$).
\If{$e=(i,n+1)$, $(1\le i\le n)$ is an edge with minimum weight incident on the vertex $n+1$} $\cG_{T}\gets \cH_T \cup e$.
\EndIf
\State\Return $\cG_{T}$.
\end{algorithmic}
\end{algorithm}
The minimum spanning tree in Step 1 of Algorithm 1 can be obtained using any of the minimum spanning tree algorithms in the literature \cite{West}.
If Prim's algorithm \cite{West} with time complexity $O(|\cE|\log|\cV|)$ or $O(|\cE|+|\cV|\log|\cV|)$ (depending on the data structure used) is employed in Step 1, then Algorithm  1 has time complexity $O\big(|\cE|\log|\cV|+$deg$(n+1)\big)$ or $O\big(|\cE|+|\cV|\log|\cV|+$deg$(n+1)\big)$, where deg$(n+1)$ is the degree of the $(n+1)^{\rm th}$ vertex.

\section{Conclusion}
\label{sec:conclusion}
We obtained graph theoretic necessary and sufficient conditions for  structural controllability and accessibility of bilinear systems on $\SE(n)$.
Furthermore, we studied the problem of finding a sparsest pattern for structural controllability. 
Using these results and the notion of $k$-edge connectedness \cite{West} of graphs, one can obtain equivalent conditions for structural controllability of bilinear systems on $\SE(n)$ under $k$ link failures; we avoid giving explicit details for the sake of brevity.

\appendix
\section{Appendix}
\subsection*{Auxiliary Results}
The following results are used in the proof of Theorem 3.1. 
\begin{lemma}
\label{lem:identity-operators}
Let $\cD^{(0)} =\cB_\Lambda$ and $\cG^{(0)} =\cG\big(\cB_\Lambda\big)$, where $\cB_\Lambda\subseteq \cSE(n)$. 
For $i \geq 0$, 
\begin{equation*}
\cB\big(\cG\big(\cD^{(i)}\big)\big) = \cD^{(i)} \quad\mbox{and}\quad \cG\big(\cB\big(\cG^{(i)}\big) \big) = \cG^{(i)}.
\end{equation*}
\end{lemma} 
Lemma \ref{lem:identity-operators} follows from Definition \ref{def:pattern-graph} and \eqref{eq:subspace-B}. 
In particular, it follows from the construction of $\cG$ and $\cB$ that $\cB\circ \cG$ and $\cG \circ \cB$ act as identity operators on $\cD^{(i)}$ and $\cG^{(i)}$, respectively. 

The following lemma expounds the relation between transitive closures and derived distributions corresponding to $\cB_\Lambda\subseteq \cSE(n)$.
\begin{lemma}
\label{lem:core-lemma}
Let $\cD^{(0)} =\cB_\Lambda$ and $\cG^{(0)} =\cG\big(\cB_\Lambda\big)$, where $\cB_\Lambda\subseteq \cSE(n)$.
Then, $\cG\big(\cD^{(i)}\big) = \cG^{(i)}$ and $\cB\big(\cG^{(i)}\big) = \cD^{(i)}$ for $i \geq 0$.
Moreover, $\cD^{(n)}=\overline{\cD}$.
\end{lemma}
\begin{proof}
Using Definition \ref{def:pattern-graph} and \eqref{eq:subspace-B}, we get $\cG\big(\cD^{(0)}\big)=\cG^{(0)}$ and $\cB\big(\cG^{(0)}\big)=\cD^{(0)}$, which is the base step of induction. 
Next, we show that if the assertion holds for $i=i_1 \in \N$, then it holds for $i=i_1+1$.
To show that $\cG\big(\cD^{(i_1+1)}\big)=\cG^{(i_1+1)}$, we first show that $\cG^{(i_1+1)}\subseteq \cG\big(\cD^{(i_1+1)}\big)$.
Let $(p,r)\in\cG^{(i_1+1)} \setminus \cG^{(i_1)}$.
Since $\cG^{(i_1+1)}$ is the $1$-step transitive closure of $\cG^{(i_1)}$, there exists $q\in \cV$ such that $(p,q),(q,r)\in\cG^{(i_1)}=\cG\big(\cD^{(i_1)}\big)$. As explained below, it can be shown by contradiction that $q \neq n+1$. If $q=n+1$, then $(p,q), (q,r) \in \cE^{(i_1)}_b$. Consequently, by Definition \ref{def:transitive-closure}, $(p,r)\notin \cG^{(i_1+1)}$ which is a contradiction. Therefore, without loss of generality, we have the following cases.
\begin{description}
\item[Case I:] $p,q,r<n+1$.
By the induction hypothesis, $\cB\big(\cG^{(i_1)}\big)=\cD^{(i_1)}$; therefore, in this case, $\tilde{\Omega}_{pq},\tilde{\Omega}_{qr}\in \cD^{(i_1)}$.
Now, it follows from \eqref{eq:ELA-LB-1} and \eqref{eq:derived-distribution} that $\tilde{\Omega}_{pr}\in \cD^{(i_1+1)}$; thus, $(p,r)\in \cG\big(\cD^{(i_1+1)}\big)$.
\item[Case II:] $p,q < n+1$ and $r=n+1$.
It can be shown on similar lines to Case I that $E_{p(n+1)}\in \cD^{(i_1+1)}$, which implies that $(p,n+1)\in \cG\big(\cD^{(i_1+1)}\big)$. 
\end{description}
Next, if $(p,r) \in \cG^{(i_1+1)} \cap \cG^{(i_1)}$, then by the induction hypothesis, $\tilde{\Omega}_{pr}\in \cD^{(i_1)}\subseteq \cD^{(i_1+1)}$ which implies that $(p,r)\in \cG\big(\cD^{(i_1+1)}\big)$. This concludes that $\cG^{(i_1+1)}\subseteq \cG\big(\cD^{(i_1+1)}\big)$.

To show that $\cG\big(\cD^{(i_1+1)}\big)\subseteq \cG^{(i_1+1)}$, we use similar arguments as above.
In particular, if $(p,r)\in \cG\big(\cD^{(i_1+1)}\big)$ and $p,r<n+1$, then $\tilde{\Omega}_{pr}\in \cD^{(i_1+1)}$.
Now, if $\tilde{\Omega}_{pr}\in \cD^{(i_1)}$, then $(p,r)\in \cG^{(i_1)}$ which implies that $(p,r)\in\cG^{(i_1+1)}$.
Whereas if $\tilde{\Omega}_{pr} \notin \cD^{(i_1)}$, then it follows from the construction of $\cD^{(i_1+1)}$ (see \eqref{eq:derived-distribution}) and \eqref{eq:ELA-LB} that there exist $q \in \{1,\ldots,n\}$ such that $\tilde{\Omega}_{pq},\tilde{\Omega}_{qr}\in \cD^{(i_1)}$ and $\big[\tilde{\Omega}_{pq}, \tilde{\Omega}_{qr}\big] = \tilde{\Omega}_{pr}$.
Consequently, by the induction hypothesis and Definition \ref{def:transitive-closure}, we can conclude that $(p,q),(q,r)\in\cG^{(i_1)}$ and $(p,r)\in \cG^{(i_1+1)}$.
The case $(p,r)\in \cG\big(\cD^{(i_1+1)}\big)$ for $p<n+1$ and $r=n+1$ can be proved on similar lines. Therefore, we have $\cG^{(i_1+1)} = \cG\big(\cD^{(i_1+1)}\big)$. 
Subsequently, using Lemma \ref{lem:identity-operators}, 
\begin{equation*}
\cB\big(\cG^{(i_1+1)}\big)=\big(\cB\circ\cG\big)\big(\cD^{(i_1+1)}\big)=\cD^{(i_1+1)}.
\end{equation*}

Recall from Definition \ref{def:transitive-closure} that $\cG^{(n)}$ is the transitive closure of $\cG^{(0)}$. Thus, $\cG^{(n+j)}=\cG^{(n)}$ for all $j \geq 1$.
Now, since $\cG^{(n)}=\cG^{(n+1)}$, we have $\cB\big(\cG^{(n)}\big)=\cB\big(\cG^{(n+1)}\big)$; equivalently, $\cD^{(n)}=\cD^{(n+1)}$.
As a result, $\cD^{(n)}=\overline{\cD}$. 
\end{proof}

\begin{remark}
\label{rmk:auxiliary-results}
Lemmas \ref{lem:identity-operators} and \ref{lem:core-lemma} are in fact counterparts of results in \cite[Sec. III]{Tsope}, where directed graphs are associated with pattern matrices over $\cGL(n)$ and $\cSL(n)$.
However, the notion of solid and broken edges in Definition \ref{def:pattern-graph} play a pivotal role here in determining the transitive closure of $\cG({\cB_{\Lambda}})$. 
Whereas this idea is not needed for the analysis done in \cite{Tsope} due to the structure of $\cGL(n)$ and $\cSL(n)$.
Thus, Lemmas \ref{lem:identity-operators} and \ref{lem:core-lemma} are not immediate consequences of \cite[Sec. III]{Tsope}. 
More specifically, when $\cB_{\Lambda}\subseteq \cSE(n)$, ideas in \cite{Tsope} will not work directly in the construction of $\cG(\cB_{\Lambda})$.
\end{remark}


\bibliographystyle{elsarticle-num}        
\bibliography{References}           



\end{document}